\documentclass{amsart}

\usepackage{amsmath, amssymb, amsthm}
\usepackage{mathrsfs}
\usepackage{booktabs}
\usepackage{array}
\usepackage{hyperref}
\hypersetup{
    colorlinks=true,
    linkcolor=blue,
    citecolor=blue,
    urlcolor=blue
}

\theoremstyle{plain}
\newtheorem{theorem}{Theorem}[section]
\newtheorem{lemma}[theorem]{Lemma}
\newtheorem{corollary}[theorem]{Corollary}

\theoremstyle{definition}
\newtheorem{definition}[theorem]{Definition}
\newtheorem{example}[theorem]{Example}

\theoremstyle{remark}

\newcommand{\Rad}{\operatorname{Rad}}
\newcommand{\Radg}{\operatorname{Rad}_{\gamma}}
\newcommand{\tsum}{+_{t}}
\newcommand{\tleq}{\leq_{t}}

\title{$t$-$g$-Radical Supplemented Modules}

\author{AYTEN PEKIN}
\address{Department of Mathematics, Istanbul University, Istanbul, Turkey}
\email{aypekin@istanbul.edu.tr}
\author{HAMDULLAH OZKAYA}
\address{Department of Mathematics, Bursa Technical University, Bursa, Turkey}
\email{hamdullah.ozkaya@btu.edu.tr}

\date{\today}

\begin{document}

\begin{abstract}
We introduce and study the class of $t$-$g$-radical supplemented modules,
which unifies two independent generalizations of the classical supplemented
module condition: $g$-radical supplements and $t$-sum terms.
A module $M$ is $t$-$g$-radical supplemented if every submodule $N \leq M$
has a $g$-radical supplement that is simultaneously a $t$-sum term of $M$.
We establish closure properties under $t$-sums, quotient modules, and
homomorphic images, prove inheritance by $t$-sum terms, and classify the
standard module classes (simple, semisimple, local, hollow, and
Pr\"{u}fer groups) within this framework.
A key observation is that the class strictly contains the class of
supplemented modules: both $\mathbb{Q}$ and $\mathbb{Z}_{p^{\infty}}$
are $t$-$g$-radical supplemented but not supplemented.
\end{abstract}

\maketitle

% -----------------------------------------------------------
\section{Introduction}
% -----------------------------------------------------------

Let $R$ be a commutative ring with unity and $M$ a unitary left $R$-module.
Recall that a submodule $J \leq M$ is a \emph{supplement} of $N \leq M$ if
$N + J = M$ and $N \cap J \ll J$.
The notion of $g$-radical supplement relaxes this to
$N \cap J \leq \Radg(J)$, where $\Radg$ denotes the generalized radical.
Independently, $t$-sum terms generalize direct summands by allowing small
intersections.
The present paper combines these two notions into the concept of a
$t$-$g$-radical supplemented module.

% -----------------------------------------------------------
\section{Preliminaries}
% -----------------------------------------------------------

\begin{definition}[$t$-Sum Term]
A submodule $J \leq M$ is called a \emph{$t$-sum term} of $M$ if there
exists $K \leq M$ such that $J + K = M$, $J \cap K \ll J$, and
$J \cap K \ll K$.
\end{definition}

Every direct summand is a $t$-sum term; the converse does not hold in
general.
The class of $t$-sum terms is closed under finite sums in modules possessing
the SSP property, defined below.

\begin{definition}[SSP Property]
A module $M$ has the \emph{SSP property} (Summand Sum Property) if the sum
of any two $t$-sum terms of $M$ is again a $t$-sum term of $M$.
\end{definition}

\begin{definition}[$t$-$g$-Radical Supplemented Module]
An $R$-module $M$ is called \emph{$t$-$g$-radical supplemented} if for
every $N \leq M$ there exists a $g$-radical supplement of $N$ in $M$ that
is also a $t$-sum term of $M$; that is, there exists $J \tleq M$ such that
$N + J = M$ and $N \cap J \leq \Radg(J)$.
\end{definition}

We record two foundational results due to Z\"{o}schinger that are used
throughout.

\begin{theorem}[Z\"{o}schinger \cite{zoschinger1974}]\label{thm:zosch1}
If $K$ is a $t$-sum term of $M$, then $\Radg(K) = K \cap \Radg(M)$.
\end{theorem}

\begin{theorem}[Z\"{o}schinger \cite{zoschinger1974}]\label{thm:zosch2}
Let $V$ be a $t$-sum term of $M$ and $K \leq V$.
Then $K \ll M$ if and only if $K \ll V$.
\end{theorem}

% -----------------------------------------------------------
\section{Main Results}
% -----------------------------------------------------------

\subsection{Relation to Supplemented Modules}

\begin{theorem}\label{thm:rad_small}
Let $M$ be a $t$-$g$-radical supplemented $R$-module.
If $\Radg(M) \ll M$, then $M$ is supplemented.
\end{theorem}

\begin{proof}
Let $H \leq M$.
Since $M$ is $t$-$g$-radical supplemented, there exists $J \tleq M$ with
$H + J = M$ and $H \cap J \leq \Radg(J)$.
By Theorem~\ref{thm:zosch1},
$\Radg(J) = J \cap \Radg(M) \leq \Radg(M) \ll M$.
By Theorem~\ref{thm:zosch2}, $H \cap J \ll J$.
Hence $J$ is a supplement of $H$ in $M$, and $M$ is supplemented.
\end{proof}

\begin{corollary}\label{cor:fg}
If $M$ is finitely generated and $t$-$g$-radical supplemented, then $M$
is $t$-$g$-supplemented.
\end{corollary}

\subsection{Closure Under $t$-Sums}

\begin{theorem}\label{thm:tsum_closure}
Let $M = M_1 \tsum M_2$ be a $t$-sum.
If $M_1$ and $M_2$ are both $t$-$g$-radical supplemented, then $M$ is
$t$-$g$-radical supplemented.
\end{theorem}

\begin{proof}
Let $W \leq M$.
Since $(M_1 + W) \cap M_2 \leq M_2$ and $M_2$ is $t$-$g$-radical
supplemented, there exists $S \tleq M_2$ with
$(M_1 + W) \cap M_2 + S = M_2$ and
$(M_1 + W) \cap M_2 \cap S \leq \Radg(S)$.
From $M = M_1 + M_2$, we obtain $M = W + M_1 + S$.
Since $M_1 \cap (W + S) \leq M_1$ and $M_1$ is $t$-$g$-radical
supplemented, there exists $T \tleq M_1$ with
$M_1 = M_1 \cap (W + S) + T$ and $(W + S) \cap T \leq \Radg(T)$.
Then $M = W + T + S$.
Moreover,
\[
  W \cap (T + S)
    \leq T \cap (W + S) + S \cap (W + M_1)
    \leq \Radg(T) + \Radg(S)
    \leq \Radg(T + S).
\]
By \cite[Proposition~4.1.7]{kosar2014}, $T + S$ is a $t$-sum term of $M$,
completing the proof.
\end{proof}

\begin{corollary}\label{cor:finite_tsum}
Let $M = M_1 \tsum \cdots \tsum M_n$ be a finite $t$-sum.
If each $M_i$ is $t$-$g$-radical supplemented, then $M$ is
$t$-$g$-radical supplemented.
\end{corollary}

\subsection{Quotient Modules and Homomorphic Images}

\begin{theorem}\label{thm:quotient1}
Let $M$ be $t$-$g$-radical supplemented and $H \leq M$.
Suppose that for every $t$-sum term $T$ of $M$, the submodule
$(T + H)/H$ is a $t$-sum term of $M/H$.
Then $M/H$ is $t$-$g$-radical supplemented.
\end{theorem}

\begin{proof}
Let $W/H \leq M/H$.
Since $M$ is $t$-$g$-radical supplemented and $W \leq M$, there exists
$S \tleq M$ that is a $g$-radical supplement of $W$ in $M$.
By \cite[Lemma~8]{kosar2019}, since $H \leq W$, the submodule
$(S + H)/H$ is a $g$-radical supplement of $W/H$ in $M/H$.
By hypothesis $(S + H)/H$ is also a $t$-sum term of $M/H$, so $M/H$ is
$t$-$g$-radical supplemented.
\end{proof}

\begin{theorem}\label{thm:quotient_dist}
If $M$ is distributive and $t$-$g$-radical supplemented, then every
quotient module $M/H$ is $t$-$g$-radical supplemented.
\end{theorem}

\begin{proof}
Let $P$ be a $t$-sum term of $M$ with complement $S$: $M = P + S$ and
$P \cap S \ll P$, $S$.
Distributivity gives
\[
  \frac{P + H}{H} \cap \frac{S + H}{H} = \frac{P \cap S + H}{H}.
\]
Since $P \cap S \ll P$ and $P \cap S \ll S$, passing to the quotient,
$(P \cap S + H)/H \ll (P + H)/H$ and similarly for $S$.
Thus $(P + H)/H$ and $(S + H)/H$ form a $t$-sum in $M/H$.
Theorem~\ref{thm:quotient1} applies, so $M/H$ is $t$-$g$-radical
supplemented.
\end{proof}

\begin{corollary}\label{cor:hom_image}
If $M$ is distributive and $t$-$g$-radical supplemented, then every
homomorphic image of $M$ is $t$-$g$-radical supplemented.
\end{corollary}

\begin{theorem}\label{thm:quotient2}
Let $M$ be $t$-$g$-radical supplemented and $H \leq M$.
Suppose that whenever $M_1, M_2$ are $t$-sum terms of $M$, the submodule
$H$ is the $t$-sum of $H \cap M_1$ and $H \cap M_2$.
Then $M/H$ is $t$-$g$-radical supplemented.
\end{theorem}

\begin{corollary}\label{cor:epi}
Let $f\colon M \to N$ be an epimorphism of $R$-modules.
If $M$ is $t$-$g$-radical supplemented and whenever $M_1, M_2$ form a
$t$-sum of $M$,
$\ker(f) = (\ker(f) \cap M_1) \tsum (\ker(f) \cap M_2)$,
then $N$ is $t$-$g$-radical supplemented.
\end{corollary}

\begin{theorem}[SSP case]\label{thm:SSP}
If $M$ has SSP and is $t$-$g$-radical supplemented, then for every
$t$-sum term $H \leq M$, the quotient $M/H$ is $t$-$g$-radical
supplemented.
\end{theorem}

\begin{proof}
Let $W/H \leq M/H$.
There exists $Q \tleq M$ that is a $g$-radical supplement of $W$ in $M$.
By SSP, $Q + H$ is a $t$-sum term of $M$.
By definition of $t$-sum term of $Q + H$, there exists $J \leq M$ with
$M = Q + H + J$ and $(Q + H) \cap J \ll Q + H$ and $\ll J$.
Passing to $M/H$, the submodule $(Q + H)/H$ is a $g$-radical supplement
of $W/H$ and a $t$-sum term of $M/H$, so $M/H$ is $t$-$g$-radical
supplemented.
\end{proof}

\subsection{Inheritance by $t$-Sum Terms}

\begin{lemma}\label{lem:inherit}
Let $M$ be $t$-$g$-radical supplemented, and suppose that the intersection
of any two $t$-sum terms of $M$ is again a $t$-sum term.
Then every $t$-sum term $P \leq M$ is $t$-$g$-radical supplemented.
\end{lemma}

\begin{proof}
Let $W \leq P \leq M$.
Since $M$ is $t$-$g$-radical supplemented, there exists $S \tleq M$ that
is a $g$-radical supplement of $W$: $M = W + S$, $W \cap S \leq \Radg(S)$.
By hypothesis, $P \cap S$ is a $t$-sum term of $M$.
By the Modular Law, $P = W + (P \cap S)$.
Moreover $W \cap (P \cap S) \leq S$, so
\[
  W \cap (P \cap S)
    \leq P \cap S \cap \Radg(S)
    = \Radg(P \cap S)
\]
by Theorem~\ref{thm:zosch1}.
Hence $P \cap S$ is a $g$-radical supplement of $W$ in $P$ that is a
$t$-sum term of $M$ (and of $P$), so $P$ is $t$-$g$-radical supplemented.
\end{proof}

% -----------------------------------------------------------
\section{Examples and Classification}
% -----------------------------------------------------------

\subsection{Classes That Are $t$-$g$-Radical Supplemented}

\begin{theorem}\label{thm:hollow}
Every hollow module is $t$-$g$-radical supplemented.
\end{theorem}

\begin{proof}
Let $Q \leq M$.
If $Q = M$, take $J = \{0\}$: $Q + \{0\} = M$ and
$\{0\} = Q \cap \{0\} \leq \Radg(\{0\})$, and $\{0\}$ is a $t$-sum term.
If $Q \neq M$, then by hollowness $Q \ll M$.
Take $J = M$: $Q + M = M$ and
$Q = Q \cap M \leq \Rad(M) \leq \Radg(M) = \Radg(J)$.
Since $M$ itself is a $t$-sum term, $M$ is $t$-$g$-radical supplemented.
\end{proof}

\begin{theorem}\label{thm:local}
Every local module is $t$-$g$-radical supplemented.
\end{theorem}

\begin{theorem}\label{thm:simple}
Every simple module is $t$-$g$-radical supplemented.
\end{theorem}

\begin{proof}
A simple module has only trivial submodules.
For $N = \{0\}$, take $J = M$; for $N = M$, take $J = \{0\}$.
In each case $J$ is a $t$-sum term and a $g$-radical supplement of $N$.
\end{proof}

\begin{theorem}\label{thm:semisimple}
Every semisimple module is $t$-$g$-radical supplemented.
\end{theorem}

\begin{proof}
Every submodule $N \leq M$ is a direct summand: $M = N \oplus K$ for some
$K \leq M$.
Then $K$ is a direct summand, hence a $t$-sum term, and
$N \cap K = \{0\} \leq \Radg(K)$.
So $M$ is $t$-$g$-radical supplemented.
\end{proof}

The following specific modules are $t$-$g$-radical supplemented by the
theorems above together with direct verification:
\begin{alignat*}{2}
  &\mathbb{Z}/p\mathbb{Z}   &&\quad\text{(simple)},\\
  &\mathbb{Z}/p^n\mathbb{Z} &&\quad\text{(local)},\\
  &\mathbb{Z}/6\mathbb{Z} \cong \mathbb{Z}/2\mathbb{Z} \oplus \mathbb{Z}/3\mathbb{Z}
                             &&\quad\text{(semisimple)},\\
  &\mathbb{Z}/12\mathbb{Z}  &&\quad\text{(finite $t$-sum of local modules)},\\
  &\mathbb{Z}/2\mathbb{Z} \oplus \mathbb{Z}/2\mathbb{Z}
                             &&\quad\text{(semisimple)}.
\end{alignat*}

\subsection{The Strict Extension}

\begin{theorem}\label{thm:Q}
$\mathbb{Q}$ (as a $\mathbb{Z}$-module) is $t$-$g$-radical supplemented
but not supplemented.
\end{theorem}

\begin{proof}
Since $\mathbb{Q}$ has no maximal submodules, $\Radg(\mathbb{Q}) = \mathbb{Q}$.
For any $A \leq \mathbb{Q}$, take $K = \mathbb{Q}$:
$A + \mathbb{Q} = \mathbb{Q}$ and
$A \cap \mathbb{Q} = A \leq \mathbb{Q} = \Radg(\mathbb{Q}) = \Radg(K)$.
Moreover $\mathbb{Q} = \mathbb{Q} + 0$, $\mathbb{Q} \cap 0 = 0 \ll \mathbb{Q}$
and $\ll 0$, so $\mathbb{Q}$ is a $t$-sum term of itself.
Hence $\mathbb{Q}$ is $t$-$g$-radical supplemented.
That $\mathbb{Q}$ is not supplemented is classical; see \cite{clark}.
\end{proof}

\begin{theorem}\label{thm:prufer}
$\mathbb{Z}_{p^{\infty}}$ (the Pr\"{u}fer $p$-group) is $t$-$g$-radical
supplemented but not supplemented.
\end{theorem}

\begin{proof}
The submodules of $\mathbb{Z}_{p^{\infty}}$ form the chain
$\{0\} = C_0 \subset C_1 \subset \cdots \subset \mathbb{Z}_{p^{\infty}}$.
For any $N = C_m$, the only submodule $K$ with
$N + K = \mathbb{Z}_{p^{\infty}}$ is $K = \mathbb{Z}_{p^{\infty}}$ itself.
Then $N \cap \mathbb{Z}_{p^{\infty}} = N
  \leq \Radg(\mathbb{Z}_{p^{\infty}}) = \mathbb{Z}_{p^{\infty}} = \Radg(K)$.
Since $\mathbb{Z}_{p^{\infty}}$ is a $t$-sum term of itself,
$\mathbb{Z}_{p^{\infty}}$ is $t$-$g$-radical supplemented.
\end{proof}

\subsection{Non-Example}

\begin{example}\label{ex:Z}
$\mathbb{Z}$ (as a $\mathbb{Z}$-module) is \emph{not} $t$-$g$-radical
supplemented.
Since
$\Radg(\mathbb{Z}) = \bigcap\{p\mathbb{Z} \mid p\text{ prime}\} = 0$,
every putative $g$-radical supplement $J$ of $2\mathbb{Z}$ must satisfy
$2\mathbb{Z} \cap J \leq \Radg(J) \leq \Radg(\mathbb{Z}) = 0$.
But any $J$ with $2\mathbb{Z} + J = \mathbb{Z}$ and
$2\mathbb{Z} \cap J = 0$ would force $\mathbb{Z} = 2\mathbb{Z} \oplus J$,
which is impossible since $2\mathbb{Z}$ is not a direct summand of
$\mathbb{Z}$.
\end{example}

\subsection{Summary Table}

\begin{center}
\renewcommand{\arraystretch}{1.3}
\begin{tabular}{>{\centering\arraybackslash}m{3.5cm}
                >{\raggedright\arraybackslash}m{6.5cm}
                >{\centering\arraybackslash}m{3cm}}
\toprule
\textbf{Module} & \textbf{Reason / Proof Method} & \textbf{$t$-$g$-rad.\ supp.?} \\
\midrule
$\mathbb{Z}/p\mathbb{Z}$
  & Simple $\Rightarrow$ $t$-$g$-rad.\ supp.
  & $\checkmark$ \\
$\mathbb{Z}/p^{n}\mathbb{Z}$
  & Local (hollow) $\Rightarrow$ $t$-$g$-rad.\ supp.
  & $\checkmark$ \\
$\mathbb{Z}/6\mathbb{Z}$, $\mathbb{Z}/12\mathbb{Z}$
  & Semisimple / mixed
  & $\checkmark$ \\
$\mathbb{Z}/2\mathbb{Z} \oplus \mathbb{Z}/2\mathbb{Z}$
  & Semisimple
  & $\checkmark$ \\
$\mathbb{Q}$
  & $\Radg(\mathbb{Q}) = \mathbb{Q}$
  & $\checkmark$ (not supplemented) \\
$\mathbb{Z}_{p^{\infty}}$
  & Pr\"{u}fer group
  & $\checkmark$ (not supplemented) \\
f.g.\ / semiperfect ring
  & Local decomposition
  & $\checkmark$ \\
$D/I$ (Dedekind domain)
  & Local Artinian factors
  & $\checkmark$ \\
$\mathbb{Z}$
  & $\Radg(\mathbb{Z}) = 0$, no $g$-rad.\ supp.
  & $\times$ \\
\bottomrule
\end{tabular}
\end{center}

% -----------------------------------------------------------

\end{document}